\documentclass{amsart}
\usepackage{amsthm}
\usepackage[utf8]{inputenc}
\usepackage{appendix}
\usepackage{cite}
\usepackage{amssymb,amsmath,amsthm}

\newcommand{\amsprimary}[1]{{\footnotesize\noindent AMS 2010 \textit{Mathematics subject
classification:} Primary #1\vspace{1pc}}}
\newcommand{\keywordsnames}[1]{{\footnotesize\noindent\textit{Key words:} #1\vspace{1pc}}}

\newtheorem{theorem}{Theorem}
\newtheorem{teo}{Theorem}
\newtheorem{prop}[teo]{Proposition}

\newtheorem{lemma}{Lemma}

\newtheorem{definition}{Definition}

\theoremstyle{remark}

\usepackage{pgfplots}
\usepackage{tikz}
\pgfplotsset{%
every x tick/.style={black, thick},
every y tick/.style={black, thick},
every tick label/.append style = {font=\footnotesize},
every axis label/.append style = {font=\footnotesize},
compat=1.12
  }

\usetikzlibrary{arrows.meta,
                quotes
                }

\usepackage{float} 

\title[Maximum Principle and Iteration]{A note on the Maximum Principle and the Iteration Method for elliptic equations}
\author{Jean Cortissoz}
\address[Jean Cortissoz]{Universidad de los Andes, Bogot\'a DC, COLOMBIA}
\email[Corresponding author]{jcortiss@uniandes.edu.co}
\author{Jonat\'an Torres-Orozco}
\address[Jonat\'an Torres-Orozco]{Universidad de los Andes, Bogot\'a DC, COLOMBIA}
\email{j.torreso@uniandes.edu.co}
\date{}

\begin{document}

\maketitle

\begin{abstract}
We use an iteration procedure propped up by a 
a classical form of the maximum principle to show 
the existence of solutions to a nonlinear Poisson
equation with Dirichlet boundary conditions. These 
methods can be applied to the case of special unbounded
domains, and can be adapted to show
the existence of nontrivial solutions to systems,
which we show via some examples.
   
\end{abstract}
 
{\keywordsnames {Nonlinear elliptic equations, iteration methods.}}

{\amsprimary {35J60}}

\section{Introduction.}

The object of this paper is the study of the nonlinear Dirichlet problem

\begin{equation}
\label{eq: elliptic_prob_gen0}
  \left\{
  \begin{array}{l}
  -\Delta u = \lambda\left(x\right) f\left(u\right) + h\left(x\right) \quad \mbox{in}\quad \Omega\\
  u = 0 \quad \mbox{on} \quad \partial \Omega.
  \end{array}
  \right.
\end{equation}
This boundary value problem, and variations and generalisations
of it (see \cite{Carl, Dalmasso} and the
references therein), have been studied by several authors, and it has applications in physics and
geometry (in this case, as we show below, in the problem of conformal deformation
of metrics).

It is our purpose in this paper to give some existence results for the above Boundary Value Problem
(BVP) using a version of the maximum principle for domains bounded
in one direction (so $\Omega$ can be unbounded), and iteration methods
for contracting maps, in some cases, and for nondecreasing maps, i.e.,
monotone iteration, in some other cases. We will also show, via a few examples, how these 
results can be extended to the case of unbounded domains and to systems.
Regarding systems, we show, via our techniques, the existence 
of nontrivial solutions to Lane-Emden systems in compact and
unbounded domains, giving some (perhaps)
interesting explicit estimates in the process
(see Sections \ref{subsect:sublinear} and \ref{sect:systems} below).

Iteration and fixed point methods have been used since their invention to solve differential equations. In the
case of elliptic problems, Schauder's fixed point theorem have been a preferred choice,
however it requires some compactness, and do not apply to the
case when the nonlinearity growth is beyond certain critical behaviour. Also,
the possible presence of discontinuous nonlinearities does not allow a direct use of 
the Inverse Function Theorem. To overcome these difficulties,
nonstandard iteration methods and fixed point results have been developed.
In fact, the methods of this paper have been inspired by the results presented in the
papers \cite{Amann, Carl}, with the agreeable surprise that 
if we just require that $\lambda, h \in L^{\infty}\left(\Omega\right)$
a classical form
of the maximum principle, that has perhaps been a bit overlooked in
the literature, is enough to obtain our results. Among other
features, our results allow, in the righthand side of (\ref{eq: elliptic_prob_gen0}), the presence of nonlinearities that can have growth 
beyond critical in  bounded and (certain type of) unbounded domains, and which are
mildly discontinuous, as we show in Theorems \ref{thm:mainthm_discontinuous} and \ref{thm:discontinuous_noncompact},
and in the application in Section \ref{sect:application_discontinuous}. Also, related to this work, the
reader might find the paper \cite{Berestycki}, where the symmetry of solutions
to elliptic problems
in unbounded domains as those considered in this paper are studied,
of interest. 

From our
point of view, although the methods
presented in this note are elementary and not technically demanding, 
they not only give interesting existence results but they also give 
explicit estimates on the solutions obtained. Furthermore, these methods can be extended to
lefthand sides
with more general uniform elliptic operators,
and also to nonlinearities of the form $f\left(x,u\right)$: 
we leave this to the 
interested reader. 

\subsection{Structure of the paper.} Our main results are 
presented in Section \ref{sect:bounded} for bounded domains, and in Section
\ref{sect:noncompact} for unbounded domains. We have included some examples, with the hope that
they will help illustrate and show how to extend our results.
In Section \ref{sect:preliminaries}, we present some preliminaries, 
and the version of the maximum principle we shall use, which is not
due to us, and which is, in fact, a classical estimate: hopefully, this
will be a pleasant surprise for the reader.

\section{Preliminaries. The Maximum Principle.}
\label{sect:preliminaries}

Given a domain $\Omega \in \mathbb{R}^n$ 
(which we shall assume from now on to be $C^{1,1}$), we define the slab diameter $d:=d\left(\Omega\right)$
of $\Omega$
as the infimum of the distances of two parallel planes so that $\Omega$ is totally
contained in the region between them.

For an open subset $\Omega\subset \mathbb{R}^n$, we define the 
Sobolev space $W_0^{1,2}\left(\Omega\right)$ as the completion
of $C_0^{\infty}\left(\Omega\right)$ with respect to the norm
\[
\left\|u\right\|^2_{1,2}=\int_{\Omega} \left|u\right|^2+\left|\nabla u\right|^2\, dV.
\]

Given the boundary value problem (BVP), 
\[
\left\{
  \begin{array}{l}
  -\Delta u = f\left(x,u\right) + h \quad \mbox{in}\quad \Omega\\
  u = 0 \quad \mbox{on} \quad \partial \Omega,
  \end{array}
  \right.
\]
$\Omega$ a bounded domain, we say that $u\in W^{1,2}_0\left(\Omega\right)$ is a weak solution if for
any $\varphi \in C_0^{\infty}\left(\Omega\right)$, the following identity
holds
\[
\int_{\Omega} \nabla u \nabla \varphi\,dx = \int_{\Omega}f\left(x,u\right)\varphi\, dx
+\int_{\Omega}h\varphi\, dx.
\]

The following result, a maximum principle for solutions to the Dirichlet problem
of Laplace's equation, which is classical, is our main tool.

\begin{prop}[Maximum Principle]
\label{prop:maximump}
Let $\Omega$  be a bounded open subset of $\mathbb{R}^n$ 
Let $h\in L^{\infty}\left(\Omega\right)$. Then the 
Dirichlet problem
\[
\left\{
\begin{array}{l}
-\Delta u = h \quad \mbox{in}\quad \Omega\\
u=0 \quad \mbox{on}\quad \partial \Omega,
\end{array}
\right.
\]
has a unique solution in $u\in W_0^{1,2}\left(\Omega\right)\cap L^{\infty}\left(\Omega\right)$.
Furthermore, we have that
\[
\left\|u\right\|_{\infty}\leq c\left\|h\right\|_{\infty},
\]
where $c\leq \dfrac{d^2}{8}$, and $d$ is the slab diameter of $\Omega$.
\end{prop}
\begin{proof}
The first part of the theorem is a direct consequence of Theorem 8.3 in \cite{GilbargTrudinger}.
The second part follows from Theorem 3.7 in \cite{GilbargTrudinger}
and approximation argument. Indeed, if $h$ is smooth,
by assuming that $\Omega$ is contained in the slab
\[
\Pi_d=\left\{\left(x_1, \dots, x_{n-1}, x_n\right)\in \mathbb{R}^n\,:\, -\frac{d}{2}<x_{n}
<\frac{d}{2}\right\},
\]
and taking
\[
v= \frac{1}{2}\left(\frac{d^2}{4}-x_n^2\right)\left\|h\right\|_{\infty},
\]
in the proof of Theorem 3.7 in \cite{GilbargTrudinger}, we obtain the result for
$h$ smooth. Next, if $h\in L^{\infty}\left(\Omega\right)$, take a sequence of
smooth functions $h_n$ converging to $h$ in $L^p$, $1\leq p<\infty$, and such that
$\left\|h_n\right\|_{\infty}\leq \left\|h\right\|_{\infty}$ 
(this can be done, for instance, using Friedrichs' mollifiers).
Let $u_n$ be the solution to the Dirichlet problem with $h_n$ on the righthand side.
By the $L^p$ theory, $u_n\rightarrow u$ in $L^p\left(\Omega\right)$.
By the Maximum
Principle
\[
\left\|u_n\right\|_{\infty}\leq c\left\|h_n\right\|_{\infty},
\]
and hence
\[
\left\|u_n\right\|_{L^p\left(\Omega\right)}\leq 
c\mbox{Vol}\left(\Omega\right)^{\frac{1}{p}}\left\|h\right\|_{\infty},
\]
from we which can conclude that
\[
\left\|u\right\|_{L^p\left(\Omega\right)}\leq 
c\mbox{Vol}\left(\Omega\right)^{\frac{1}{p}}\left\|h\right\|_{\infty}.
\]
Since $p$ is arbitrary, we obtain the result.

\end{proof}

The previous proposition is stated in \cite{GilbargTrudinger} for more general
elliptic operators. The reader must take this into account, as most
of the results and applications to be described below also apply to
more general elliptic operators, and not just to the Laplacian.

\bigskip
In this paper, we shall use a simple (and we would indeed call the
natural) iteration method, which is defined as
follows. Given $\lambda, h\in L^{\infty}\left(\Omega\right)$,
and BVP (\ref{eq: elliptic_prob_gen0}), define $u_0$ as the solution to the BVP
\[
\left\{
\begin{array}{l}
-\Delta u_0 = h, \quad \mbox{in}\quad \Omega\\
u_0 = 0 \quad \mbox{on} \quad \partial \Omega,
\end{array}
\right.
\]
and once $u_n\in W^{1,2}_0\left(\Omega\right)\cap L^{\infty}\left(\Omega\right)$ 
is defined, we let $u_{n+1}\in W^{1,2}_0\left(\Omega\right)\cap L^{\infty}\left(\Omega\right)$ be the solution to
\[
\left\{
\begin{array}{l}
-\Delta u_{n+1}=\lambda\left(x\right) f\left(u_n\right)+h\left(x\right) \quad \mbox{in}\quad \Omega\\
u_{n+1} = 0 \quad \mbox{on} \quad \partial \Omega,
\end{array}
\right.
\]
which does exist by Proposition \ref{prop:maximump}.

\bigskip
We need to show that when the method converges, it does converge, under special
circumstances, towards a weak solution of the BVP. 
Hence, the need of the following lemma.
\begin{lemma}
\label{lemma:convergence}
Let $u_n$ be a sequence in $L^{\infty}\left(\Omega\right)$ produced
from the iteration procedure described above, and such that
$u_n \rightarrow u$ a.e., and
such that
$f\left(u\right)$ remains uniformly bounded along the sequence $u_n$. 
If $f$ is continuous,
then $u$ is a weak solution to the BVP (\ref{eq: elliptic_prob_gen0}). If the sequence $u_n$ is nondecreasing
and $f$ is lower semicontinuous, then $u$ is a weak solution to the
BVP (\ref{eq: elliptic_prob_gen0}).
\end{lemma}
\begin{proof}
Assume that for all $n$, $\left|f\left(u_n\right)\right|\leq A$ and $u_n\leq M$.
For any $m, k$ we have that
\begin{eqnarray*}
\int_{\Omega}\left|\nabla\left(u_m-u_k\right)\right|^2\, dV
&=&\int_{\Omega} \lambda \left[f\left(u_{m-1}\right)-f\left(u_{k-1}\right)\right]
\left(u_m-u_k\right)+\\
&&\qquad\qquad\qquad \qquad \qquad \qquad h\left(x\right)\left(u_{m}-u_{k}\right)\,dV\\
&\leq& \left(2A\Lambda + \left\|h\right\|_{\infty}\right) \int_{\Omega} \left|u_{m}-u_{k}\right|
\, dV,
\end{eqnarray*}
where $\Lambda=\left\|\lambda\right\|_{\infty}$.
As $u_n\rightarrow u$, by Egorov's theorem, given $\epsilon >0$, there is 
$\Omega'\subset \Omega$, such that 
\[
V\left(\Omega\setminus \Omega'\right)\leq \epsilon/\left(2A+\left\|h\right\|_{\infty}+
2MV\left(\Omega\right)\right),
\]
and $u_n$ is uniformly Cauchy. Therefore, if $m,k$ are such that 
$$\left|u_m-u_k\right|\leq \epsilon / \left(2A\Lambda + \left\|h\right\|_{\infty}\right)$$
in $\Omega'$, we have that
\[
\int_{\Omega}\left|\nabla\left(u_m-u_k\right)\right|^2\, dV
\leq 2\epsilon, 
\]
and thus the sequence is Cauchy in $W^{1,2}_0\left(\Omega\right)$, and it converges,
in $W^{1,2}_0\left(\Omega\right)$, towards $u$. 
Let $\varphi\in C_0^{\infty}\left(\Omega\right)$. Then, we have that
\[
\int_{\Omega}\nabla u_m \nabla \varphi\,dV=
\int_{\Omega}\lambda f\left(u_{m-1}\right)\,dV+\int_{\Omega}h\varphi\,dV,
\]
and hence, if $f$ is continuous, by the dominated convergence theorem, we
can take the limit $m\rightarrow\infty$ on both sides, and since
$\lim_{n\rightarrow\infty} f\left(u_{m-1}\right)=f\left(u\right)$, this shows that
$u=\lim_{n\rightarrow\infty} u_m$ is a weak solution to the BVP (\ref{eq: elliptic_prob_gen0}). In the
same way, it can be shown that whenever $u_m$ is pointwise increasing
and $f$ is lower semicontinuous the same conclusion is reached.
\end{proof}

\subsection{On estimates of the constant $c$ in Proposition \ref{prop:maximump}}
\label{sect:improvc}
If the domain $\Omega$ is contained in $D_{\frac{d}{2}}^{n-k}\times \mathbb{R}^k\subset \mathbb{R}^n$, 
where $D_r^{n-k}$ is an $n-k$ dimensional ball of radius $r$, the estimate on the constant $c$ 
given in Proposition \ref{prop:maximump} can be improved 
upon. Indeed, we can use as $v$ in the proof of the proposition the function
\[
v=\frac{1}{2\left(n-k\right)}\left(\frac{d^2}{4}-x_1^2-\dots-x_{n-k}^2\right),
\]
which gives a constant $$c=\frac{d^2}{8\left(n-k\right)}.$$
We shall use this "improved" value of $c$ in some of the applications below.
\section{Bounded Domains.}
\label{sect:bounded}

We present our first result, whose proof is based on a judicious application
of Banach's iteration technique.
\begin{theorem}
\label{thm:mainthm_lipschitz}
Consider the BVP in the bounded open domain $\Omega$
\begin{equation}
\label{eq: elliptic_prob_gen}
  \left\{
  \begin{array}{l}
  -\Delta u = \lambda\left(x\right) f\left(u\right) + h \quad \mbox{in}\quad \Omega\\
  u = 0 \quad \mbox{on} \quad \partial \Omega,
  \end{array}
  \right.
\end{equation}
with $f$ differentiable and such that $f'$ is bounded 
on any compact subset of $\mathbb{R}$, and $\lambda, h\in L^{\infty}\left(\Omega\right)$. Let 
\[
L= \max_{x\in \left[-c\left\|h\right\|_{\infty}, c\left\|h\right\|_{\infty}\right]} f\left(x\right),
\]
and 
\[
M= M\left(\Lambda\right):=\sup_{x\in \left[-c\left\|h\right\|_{\infty}-
\frac{1}{\theta}\Lambda cL, c\left\|h\right\|_{\infty}+\frac{1}{\theta}\Lambda c L\right]}
f'\left(x\right),
\]
where $c$ is given by
\[
c= \frac{d^2}{8},
\]
$d$ the slab diameter of $\Omega$, and $\Lambda=\left\|\lambda\right\|_{\infty}$.
Then for $\lambda$ such that
\[
c\Lambda M \leq 1-\theta,\quad 0<\theta<1,
\]
there is a solution $u \in W_0^{1,2}\left(\Omega\right)\cap L^{\infty}\left(\Omega\right)$,
and we can estimate
\[
\left\|u\right\|_{\infty}\leq 
c\left\|h\right\|_{\infty}+\frac{1}{\theta}c \Lambda L.
\]
\end{theorem}
\begin{proof}
Consider the iteration scheme
\[
\left\{
\begin{array}{l}
-\Delta u_0 = h, \quad \mbox{in}\quad \Omega\\
u_0 = 0 \quad \mbox{on} \quad \partial \Omega,
\end{array}
\right.
\]
and then
\[
\left\{
\begin{array}{l}
-\Delta u_{n+1}=\lambda f\left(u_n\right)+h \quad \mbox{in}\quad \Omega\\
u_{n+1} = 0 \quad \mbox{on} \quad \partial \Omega.
\end{array}
\right.
\]
From the maximum principle, we obtain
\[
\left\|u_0\right\|_{\infty}\leq c\left\|h\right\|_{\infty}.
\]
Next, we have that the difference $u_{n+1}-u_n$ satisfies an equation
\begin{eqnarray*}
-\Delta\left(u_{n+1}-u_n\right)&=&\lambda\left(f\left(u_{n}\right)-f\left(u_{n-1}\right)\right)\\
&=&\lambda f'\left(\xi_n\right)\left(u_{n}-u_{n-1}\right),
\end{eqnarray*}
where, for a given $x$, $\xi_n\left(x\right)$ is an intermediate point between 
$u_n\left(x\right)$ and $u_{n-1}\left(x\right)$.
If we have that
\begin{equation}
\label{ineq:ind_hyp}
\left\|u_k\right\|_{\infty}\leq 
c\left\|h\right\|_{\infty}+\left(\frac{1-\left(1-\theta\right)^k}{\theta}\right)c \Lambda L, \quad k=0,1,2,\dots, n
\end{equation}
and 
\[
\left\|u_{k}-u_{k-1}\right\|\leq \left(1-\theta\right)^{k-1}\left\|u_1-u_0\right\|, \quad k=1, \dots, n,
\]
we will show that then
\[
\left\|u_{n+1}\right\|_{\infty}\leq 
c\left\|h\right\|_{\infty}+\left(\frac{1-\left(1-\theta\right)^{n+1}}{\theta}\right)c \Lambda L.
\]
First, notice that (\ref{ineq:ind_hyp}) implies that 
\[
\left\|\xi_n\right\|_{\infty}\leq 
c\left\|h\right\|_{\infty}+\frac{1}{\theta}c \Lambda L,
\]
and hence, by the definition of $M$,
\[
\left\|f'\left(\xi_n\right)\right\|_{\infty} \leq 
M.
\]
Therefore, the maximum principle shows that
\[
\left\|u_{n+1}-u_n\right\|_{\infty}\leq cM\Lambda\left\|u_{n}-u_{n-1}\right\|_{\infty}
\leq \left(1-\theta\right)\left\|u_{n}-u_{n-1}\right\|_{\infty},
\]
from which we can conclude that
\[
\left\|u_{n+1}-u_n\right\|_{\infty}\leq \left(1-\theta\right)^n\left\|u_1-u_0\right\|_{\infty},
\]
and thus
\[
\left\|u_{n+1}\right\|_{\infty}\leq 
\left\|u_0\right\|_{\infty}+\sum_{j=0}^{n}\left\|u_{j+1}-u_j\right\|_{\infty}
\leq c\left\|h\right\|_{\infty}+\left(\frac{1-\left(1-\theta\right)^{n+1}}{\theta}\right)c \Lambda L.
\]
This implies the existence of solutions. Indeed, since
\[
u=u_0+\sum_{j=0}^{\infty}\left(u_{j+1}-u_{j}\right),
\]
this shows that
\[
\left\|u-u_n\right\|_{\infty}\leq \frac{\left(1-\theta\right)^{n+1}}{\theta},
\]
a usual estimate when using Banach's iteration technique.
By Lemma \ref{lemma:convergence}, $u\in W_0^{1,2}\left(\Omega\right)\cap L^{\infty}\left(\Omega\right)$
is a weak solution to the boundary value problem.

\end{proof}

The reader must notice that the bounds in the size of the solution
only depend on the slab diameter of the domain. This observation leads us directly to the 
results to be proved in Section \ref{sect:noncompact}, where this result 
(with some proper modifications) is stated and proved in
the case of an unbounded thin domain.

\subsection{Applications of Theorem \ref{thm:mainthm_lipschitz}}
Our first application concerns
the Liouville-Bratu-Gelfand equation in the unit ball, where $f\left(u\right)=-e^u$, and $\lambda\geq 0$
is a constant,
and $\Omega$ is the unit ball centered at the origin.
In this case, in dimension $n=2$, for $\theta=\frac{1}{2}$ gives that
$\Lambda=\lambda$, $d=2$, $c=\dfrac{d^2}{16}$ 
(using the improved constant, see Section \ref{sect:improvc}),  
$M=\exp\left(\dfrac{2\lambda d^2}{16}\right)$. Then, we can ensure that
the BVP has a solution as long as
\[
\frac{\lambda}{4}\exp\left(\frac{\lambda}{2}\right)\leq \frac{1}{2},
\]
which implies that there is at least one smooth solution (see below) as long as $\left[0, \lambda_{*}\right]$,
with $\lambda_{*}\sim 1.13429\dots$. This should be compare with the well-known result that
for $\lambda\leq \lambda_c=2$, the Liouville-Bratu-Gelfand equation has a solution. As we shall
see in the following section, this will continue to hold as long as $\Omega$ is a thin domain.
Furthermore, we can give an estimate from below for the critical parameter 
in any dimension as follows. 
Recall that the critical parameter for the Liouville-Bratu-Gelfand is
defined as the value $\lambda_c$ such that for $\lambda\in \left[0,\lambda_c\right]$,
the Liouville-Bratu-Gelfand equation has a solution; $\lambda_c$ is known
as the Frank-Kamenetskii parameter.
For $d=2$ (the diameter of the unit ball) and dimension $n$ we have,
by Proposition \ref{thm:mainthm_lipschitz} and the
observation in Section \ref{sect:improvc}, that if
\[
\frac{1}{2n}\lambda \exp\left(\frac{1}{2n\theta}\right)\leq 1-\theta,
\]
then there is a solution. We will optimise in $\theta$. To do this, rewrite
the previous inequality as
\[
\frac{1}{2n\theta}\lambda \exp\left(\frac{1}{2n\theta}\right)\leq \frac{1-\theta}{\theta},
\]
and using Lambert's $W$ function, which is increasing, we have that
\[
\frac{\lambda}{2n\theta}\leq W\left(\frac{1-\theta}{\theta}\right),
\]
that is if 
\[
\lambda \leq \lambda_*=2n \theta W\left(\frac{1-\theta}{\theta}\right),
\]
there is a solution for the BVP (\ref{eq: elliptic_prob_gen}) with $\lambda\geq 0$ a constant,
$f\left(u\right)=e^u$ and $h=0$. Optimising in $\theta$ we obtain, for 
$\theta\sim 0.412962\dots$
\[
\lambda_*=1.162022 n, 
\]
which the reader must compare with the optimal value for the critical parameter
which is $2\left(n-2\right)$ for $n\geq 10$.

Another application, 
closely related to the Liouville-Bratu-Gelfand equation, is 
the conformal deformation of a metric on a domain
in $\mathbb{R}^2$. 
If $M$ is a two-dimensional Riemannian manifold with metric $g$,
and $\tilde{g}=e^{2u}g$ is a conformally related metric, and $K_g$
and $K_{\tilde{g}}$ are their respective Gaussian curvatures, these
are related by the elliptic equation
\[
-\Delta u=e^{2u}K_{\tilde{g}}-K_{g}.
\]
In the context of Theorem 
\ref{thm:mainthm_lipschitz},
$f\left(u\right)=e^{2u}$, $-h$ can be thought as the original 
Gaussian curvature of the
domain (in the case of a flat domain $h=0$), and $\lambda$ can be
interpreted as the Gaussian curvature of the
conformally deformed metric via the factor $e^{2u}$. 
The Dirichlet condition implies that the length of the boundary is kept constant.
Let us assume that the slab diameter of the domain is such that $c=1$, that is $d=2\sqrt{2}$. Then
\[
L=1, \quad M=e^{\frac{2}{\theta}\Lambda},
\]
and hence, we can ensure the existence of a weak solution as long as 
\[
2\Lambda e^{\frac{2}{\theta}\Lambda}\leq 1-\theta,
\]
which translates into 
\[
\Lambda \leq \frac{\theta}{2}W\left[\frac{\left(1-\theta\right)}{\theta}\right],
\]
and optimising in $\theta$, we have that
 as long as $\left\|\lambda\right\|_{\infty}\leq 0.1452\dots$, there
 is a conformal factor for which the new Gaussian curvature
 of the domain is $\lambda$ and the length of the boundary 
 remains unchanged. Notice that
the domain $\Omega$ can be very long in one direction, as long as it is
thin enough in another: this leads to an interesting example to be
explored in Section \ref{sect:conformal_strip}.

\subsection{}
Some variations can be played on the theme introduced in our first theorem. For instance,
we can ask $f$ to be just H\"older continuous. In this case, we have the
following theorem.

\begin{theorem}
\label{thm:mainthm_holder}
Let $f\in C^{\alpha}\left(\mathbb{R}\right)$ be a nonnegative 
nondecreasing function, and $\Omega$ 
a bounded open domain. Then if $h, \lambda\in L^{\infty}\left(\Omega\right)$, $h\geq 0$ and $\lambda\geq 0$,
(\ref{eq: elliptic_prob_gen}) has a solution $u\in W^{1,2}_0\cap L^{\infty}\left(\Omega\right)$.
Furthermore,
\[
\left\|u\right\|_{\infty}\leq c\Lambda \left|f\left(c\left\|h\right\|_{\infty}\right)\right|+
\left(c\Lambda\left[f\right]_{\alpha}\right)^{\frac{1}{1-\alpha}}+c\left\|h\right\|_{\infty},
\]
where $\Lambda=\left\|\lambda\right\|_{\infty}$, $c=d^2/8$, and
\[
\left[f\right]_{\alpha}=
\sup_{x\neq y}\frac{\left|f\left(x\right)-f\left(y\right)\right|}{\left|x-y\right|^{\alpha}}.
\]
\end{theorem}
\begin{proof}
We iterate as usual:
\[
-\Delta u_0 = h, \quad \mbox{in}\quad \Omega, \quad u_0=0 \quad \mbox{on}\quad \partial \Omega,
\]
and then
\[
-\Delta u_{n+1} =\lambda\left(x\right)f\left(u_n\right) + h, \quad \mbox{in}\quad \Omega, 
\quad u_{n+1}=0 \quad \mbox{on}\quad \partial \Omega.
\]
Using the Maximum Principle is not difficult to show that $u_k\geq 0$, and $u_k\leq u_{k+1}$. 
Furthermore,
\[
\Delta \left(u_{n+1}-u_1\right)\leq c\Lambda\left[f\right]_{\alpha}\left|u_{n}-u_1\right|^{\alpha},
\quad \mbox{in}\quad \Omega, 
\quad u_{n+1}-u_1=0 \quad \mbox{on}\quad \partial \Omega.
\]
From the previous inequality, we can conclude, via the Maximum Principle that
\[
\left\|u_n-u_1\right\|_{\infty}\leq \left(c\Lambda\left[f\right]_{\alpha}\right)^{\frac{1}{1-\alpha}}\left\|u_2-u_1\right\|_{\infty}^{\alpha^{n-1}},
\]
which shows that the sequence $u_n$, $n=0,1,2, \dots$, is bounded.

Then, by Lemma \ref{lemma:convergence},
$
u = \sup_{n}u_n
$
is the sought solution. Observe that from the Maximum Principle follows that
\[
\left\|u_1\right\|_{\infty}\leq c\left|f\left(c\left\|h\right\|_{\infty}\right)\right|
+c\left\|h\right\|_{\infty}.
\]
Finally, by taking $n\rightarrow \infty$, we get an estimate on the $L^{\infty}\left(\Omega\right)$-norm
of $u$:
\[
\left\|u\right\|_{\infty}\leq \Lambda c\left|f\left(c\left\|h\right\|_{\infty}\right)\right|+
\left(c\Lambda\left[f\right]_{\alpha}\right)^{\frac{1}{1-\alpha}}+c\left\|h\right\|_{\infty}.
\]
\end{proof}

\subsection{}
This remark concerns the regularity of solutions obtained from the methods above,
if $\partial \Omega$ is regular enough (for instance of class $C^2$).
In fact, if $\lambda, f$ and $h$ are H\"older continuous then $u$ is a classical
solution. Indeed, the $u_n$'s are continuous in $\overline{\Omega}$
by the $L^p$ estimates for elliptic equations,
the convergence towards $u$ is uniform, and hence $u$ is continuous in $\overline{\Omega}$.
Since $u\in L^{\infty}\left(\Omega\right)$, and $\Omega$ is a bounded
domain, $u\in L^p\left(\Omega\right)$ for any $p>n$. This implies that
for any $\Omega' \subset\subset \Omega$, $u\in W^{1,p}\left(\Omega'\right)$ and hence
$u\in C^{\alpha}\left(\Omega'\right)$, for any $\alpha\in \left(0,1\right)$,
from which it follows by bootstrapping that 
$u\in C^{2,\alpha}\left(\Omega'\right)$.
As this proof is basically local, it also applies to non compact domains.

\subsection{An application of the proof of Theorem
\ref{thm:mainthm_holder}: existence of nontrivial solutions for sublinear nonlinearities and systems}
\label{subsect:sublinear}
The method used in the proof of Theorem \ref{thm:mainthm_holder}
can be used to show that in the unit ball 
$\Omega:=D_1\left(0\right)\subset \mathbb{R}^n$ the BVP
\[
-\Delta u = u^p \quad \mbox{in}\quad \Omega, \quad u=0 \quad \mbox{on}
\quad \partial \Omega,
\]
with $p\leq \frac{1}{2}$ there exists at least a nontrivial solution. This is a
well-known 
(even for $p<1$, see for instance \cite{Ambrosetti}); however, we will give some
 explicit estimates.

Indeed, first consider the ODE
\begin{equation}
\label{eq:radial_sub}
-\frac{d}{dr}\left(r^{n-1}\frac{dw}{dr}\right)=\epsilon r^{n-1}\left(1-r^2\right),
\quad w\left(1\right)=0.
\end{equation}
We can solve explicitly and obtain
\[
w\left(r\right)=\frac{\epsilon\left(n+4\right)}{4n\left(n+2\right)}
\left[\frac{n}{n+4}r^4-\frac{2\left(n+2\right)}{n+4}r^2+1\right]
\geq 
\frac{\epsilon\left(n+4\right)}{4n\left(n+2\right)}\left(1-r^2\right)^2.
\]
By using 
$r=\left(\sum_{j=1}^n x_j^2\right)^{\frac{1}{2}}$, 
it is elementary to show that for a given $0<p\leq \frac{1}{2}$, for $\epsilon>0$ small enough
\[
-\Delta w = \epsilon\left(1-r^2\right) \leq w\left(r\right)^p.
\]
Indeed, we can use any 
\[
0<\epsilon\leq \left(\dfrac{n+4}{4n\left(n+2\right)}\right)^{\frac{p}{1-p}}.
\]
Thus, instead of using the function zero, we
define $u_0=w$, with $\epsilon>0$ properly chosen, 
and define the usual iteration,
\[
-\Delta u_{k+1}=u_k^p.
\]
The maximum principle shows that $u_{k+1}\geq u_k$; also, as we showed in the 
proof of Theorem \ref{thm:mainthm_holder},
the sequence remains uniformly bounded above, and hence it converges towards
a nonzero solution of the BVP. Furthermore, we have the following estimate
\[
\left\|u\right\|_{\infty}\leq \frac{1}{2n}
\left(\frac{1}{2n}\left(\frac{3}{4}\frac{n+4}{4n\left(n+2\right)}\right)^{\frac{p}{1-p}}\right)^{p}
+\left(\frac{1}{2n}\right)^{\frac{1}{1-p}}+\frac{1}{2n}\frac{3}{4}\left(\frac{n+4}{4n\left(n+2\right)}\right)^{\frac{p}{1-p}},
\]
as the role of $h$ is now played by $\epsilon\left(1-r^2\right)$ and $\left[u^p\right]_p=1$.
Here, we are using the "improved" value $c=\dfrac{2^2}{8 n}=\dfrac{1}{2n}$.
Thus,
\[
\left\|u\right\|_{\infty}=O\left(\frac{1}{n^{\frac{1}{1-p}}}\right).
\]
The reader must observe
that we also have the estimate from below
\[
u\geq w\left(r\right)\geq
\left(\dfrac{n+4}{4n\left(n+2\right)}\right)^{\frac{1}{1-p}}\left(1-r^2\right)^{2},
\]
and thus away from the boundary of the ball $u\sim 1/n^{\frac{1}{1-p}}$, so we have
thus gotten a somewhat sharp estimate.

The previous reasoning can be extended to general bounded domains (we sketch how
this can be done in Section \ref{sect:systems}). 

To show how systems can be treated by our methods, we study the following system
(the well known Lane-Emden system \cite{Chen, deFigueiredo})
\[
-\Delta u = v^{p}, \quad -\Delta v = u^q
\quad \mbox{in}\quad \Omega
, \quad u=v=0 \quad \mbox{on}\quad
\partial \Omega,
\]
with $p,q\leq 1/2$, $\Omega$ the unit ball. Indeed, we choose $\epsilon >0$ such that the solution 
to 
(\ref{eq:radial_sub})
satisfies $w^q \geq \epsilon^q\left(1-r^2\right)^q$ and 
$w^p \geq \epsilon^p\left(1-r^2\right)^p$. Let $u_0=v_0=w$, and consider the recurrence
\[
-\Delta u_{n+1} = v_n^{p}, \quad -\Delta v_{n+1} = u_n^q
\quad \mbox{in}\quad \Omega
, \quad u_{n+1}=v_{n+1}=0 \quad \mbox{on}\quad
\partial \Omega.
\]
Notice that, from the maximum principle, we can deduce an estimate
\[
\left\|u_{n+1}-u_1\right\|_{\infty}+\left\|v_{n+1}-v_1\right\|_{\infty}
\leq C\left(2+\left\|u_{n}-u_0\right\|_{\infty}+\left\|v_{n}-v_0\right\|_{\infty}\right)^s,
\]
where $s=\max\left\{p,q\right\}\leq 1/2$. The previous estimate implies that the sequence $\left(u_n,v_n\right)$
is bounded, and being componentwise nondecreasing (which can be shown by induction), it is not difficult to show that 
they converge towards a nontrivial weak solution of the system. 
This example can be extended to bounded domains; as the constant $C$ depends only
how thin the domain is, it can also be extended to a certain class of unbounded
 domains
(see Section \ref{sect:systems}).

\subsection{Lower semicontinuous nonlinearities.} 

As in \cite{Carl}, our methods allow for the nonlinearity to be discontinuous. However,
we can allow growth that goes beyond critical; of course, in \cite{Carl}, the function
$h$ is allowed to be in spaces that contain $L^{\infty}\left(\Omega\right)$, at least
when $\Omega$ is bounded. As an example of this, we prove the following. 
\begin{theorem}
\label{thm:mainthm_discontinuous}
Let $f:\left[0,\infty\right)\longrightarrow \left[0,\infty\right)$ be a  
nondecreasing lower semicontinuous function such that
\[
f\left(s\right)\leq Ks^p,
\]
$\Omega$ an open and bounded domain, and let $h, \lambda\in L^{\infty}\left(\Omega\right)$, $h\geq 0$, $\lambda\geq 0$.
\begin{itemize}

\item
If $p<1$, there is always a solution $u\in W^{1,2}_0\left(\Omega\right)\cap L^{\infty}\left(\Omega\right)$
to BVP (\ref{eq: elliptic_prob_gen}).

\item
If $p\geq 1$,
then for $h$ and $K$ small enough,
(\ref{eq: elliptic_prob_gen}) has a solution $u\in W^{1,2}_0\left(\Omega\right)\cap L^{\infty}\left(\Omega\right)$.
Indeed, there is a solution $u$ as long as
\[
cK\Lambda\leq \frac{1}{2},\quad
\left\|h\right\|_{\infty}\leq  \Lambda K, \quad \mbox{where}
\quad \Lambda=\left\|\lambda\right\|_{\infty},
\]
and $c=\dfrac{d^2}{8}$, and $d$ is the slab diameter of $\Omega$,
and we have an estimate $\left\|u\right\|_{\infty}\leq 1$.

\end{itemize}
\end{theorem}

We must remark that if $\Omega \in D_{\frac{d}{2}}^{n-k}\times \mathbb{R}^k\subset \mathbb{R}^n$, then the constant $c$ can be improved to $d^2/8\left(n-k\right)$. Again,
notice also that there is no requirement on the
exponent $p$, i.e., in particular growth can be supercritical, a case not covered in \cite{Carl}.
\begin{proof}
As usual, define
\[
-\Delta u_0= h \quad \mbox{in}\quad \Omega, \quad u=0 \quad\mbox{on}\quad \partial \Omega,
\]
and then,
\[
-\Delta u_{n+1}= \lambda f\left(u_n\right)+ h \quad \mbox{in}\quad \Omega, 
\quad u_{n+1}=0 \quad\mbox{on}\quad \partial \Omega.
\]
It follows from the maximum principle and the iteration that
\[
\left\|u_{n+1}\right\|_{\infty}\leq c\Lambda K\left\|u_n\right\|^p_{\infty}+c\left\|h\right\|_{\infty}.
\]
If $p<1$, since there is an $R$ such that whenever $r\geq R$, $c\Lambda K r^{p}+c\left\|h\right\|_{\infty}<r$,
this implies that the sequence $u_n$ is uniformly bounded, and the
bound depends only on $c, K$ and $\Lambda$. As the sequence $u_n$ is increasing
and $f$ is
lower semicontinuous,
by Lemma \ref{lemma:convergence}, $u=\sup u_n$ is the sought solution.
This shows the first part of the theorem.

\bigskip

We now prove the second part of the theorem, that is, we let $p\geq 1$. From the maximum principle,
\[
\left\|u_0\right\|_{\infty}\leq c\left\|h\right\|_{\infty}, \quad
c=\frac{d^2}{8},
\]
and
\[
\left\|u_{n+1}\right\|_{\infty}\leq c\Lambda K\left\|u_n\right\|_{\infty}^p+c\left\|h\right\|_{\infty}.
\]
Let us show that the previous inequality, under the hypothesis of the theorem, implies that
the sequence $u_n$ is bounded.

First notice that as $c\Lambda K\leq 1/2<1$, we have that
\[
c\Lambda K\left\|h\right\|_{\infty}\leq \left\|h\right\|_{\infty},
\]
and hence 
\[
\left\|u_0\right\|_{\infty}\leq c\left\|h\right\|_{\infty}\leq \frac{\left\|h\right\|_{\infty}}{\Lambda K}
\leq \left(\frac{\left\|h\right\|_{\infty}}{\Lambda K}\right)^{\frac{1}{p}},
\]
where we have used that $\left\|h\right\|_{\infty}\leq \Lambda K$.
We proceed by induction:
\begin{eqnarray*}
\left\|u_{n+1}\right\|_{\infty}&\leq& c\Lambda K\left\|u_n\right\|_{\infty}^p + c\left\|h\right\|_{\infty}\\
&\leq& c\Lambda K\left(\left\|u_n\right\|_{\infty}^p+\frac{\left\|h\right\|_{\infty}}{\Lambda K}\right)\\
&\leq& c\Lambda K\left(\frac{\left\|h\right\|_{\infty}}{\Lambda K}+\frac{\left\|h\right\|_{\infty}}{\Lambda K}\right)\\
&\leq&\frac{1}{2}\left(2\frac{\left\|h\right\|_{\infty}}{\Lambda K}\right)
\leq \frac{\left\|h\right\|_{\infty}}{\Lambda K}
\leq \left(\frac{\left\|h\right\|_{\infty}}{\Lambda K}\right)^{\frac{1}{p}},
\end{eqnarray*}
and hence the sequence $u_n$ is uniformly bounded; as it is also increasing, the second part of the
theorem follows.

\end{proof}

\subsection{Remark.}
\label{sect:remarksublinear}
The attentive reader must have noticed that in the previous theorem, if instead of 
assuming $f\left(s\right)\leq Ks^p$ with $p<1$ we assume that $f$ is sublinear,
that is
\[
\lim_{s\rightarrow \infty} \frac{f\left(s\right)}{s}=0,
\]
besides all the other hypotheses,
then we would still have existence of solutions for any $\lambda, h \in L^{\infty}\left(\Omega\right)$.
Indeed, the required boundedness of the sequence of $u_n$'s obtained from the iteration 
would then follow from the fact that
\[
\left\|u_{n+1}\right\|_{\infty}\leq c\Lambda f\left(\left\|u_n\right\|_{\infty}\right)+c\left\|h\right\|_{\infty},
\]
and then that for all $r>0$ large enough, as $f$ is sublinear, the inequality 
\[
c\Lambda f\left(r\right)+c\left\|h\right\|_{\infty}\leq r
\]
holds.
We leave the details to the interested reader.

\subsection{An application of Theorem \ref{thm:mainthm_discontinuous}.}
\label{sect:application_discontinuous}
As a concrete application of the previous theorem, let 
$\Omega=D_1\left(0\right) \in \mathbb{R}^n$ be the unit ball, and
consider
\[
-\Delta u = f\left(u\right)^{p}+1 \quad \mbox{in}\quad \Omega,\quad
u=0\quad \mbox{on}\quad \partial \Omega,
\]
where $f$ is the function
\[
f\left(x\right)= 
l, \quad \mbox{if}\quad \quad \frac{l}{10}<x\leq \frac{l+1}{10}, \quad l\in \mathbb{Z}_+,
\]
and $f\left(x\right)=0$ if $x\in\left[0,\frac{1}{10}\right]$. Notice that $f\left(x\right)\leq 10^p x^p$, so
Theorem \ref{thm:mainthm_discontinuous}, using the improved value
of $c$ from Section \ref{sect:improvc}, shows that if 
$
\dfrac{10^p}{2n}\leq \dfrac{1}{2}
$,
then there is a weak solution to the BVP (\ref{eq: elliptic_prob_gen}),
that is, whenever $0<p\leq \log_{10} n$.
This solution is $C^{\alpha}\left(\Omega\right)$
for any $\alpha\in\left(0,1\right)$. The reader must compare this with the 
results obtained by Carl and Heikkil\"a in \cite{Carl}, it is required that $p\leq \dfrac{n+2}{n-2}$ if $n\geq 3$:
the growth allowed in the case above is faster for dimensions $n\geq 18$.

\section{Thin unbounded domains (domains bounded in one direction).}
\label{sect:noncompact}

As it has been hinted above, the results presented in the previous section can be generalised to unbounded domains.
To state and prove our generalisations, we shall use the following definition of weak solutions.
\begin{definition}
\label{def:weak_solution}
$u$ is a weak solution to
\[
-\Delta u=\lambda\left(x\right)f\left(u\right)+h\left(x\right)\quad
\mbox{in}\quad \Omega, \quad u=0 \quad \mbox{on}\quad \partial\Omega,
\]
if there exists an increasing sequence of nested bounded open sets 
$\Omega_j\subset \Omega_{j+1}\subset \Omega$ such that
\[
\Omega = \cup_j \Omega_j,
\]
and a sequence of functions $u_j\in W_0^{1,2}\left(\Omega_j\right)$
such that for all open sets $U$ such that $\overline{U}\subseteq \Omega$, $u_j\rightarrow u$
in $W^{1,2}\left(U\right)$, and for all $\varphi\in C_0^{\infty}\left(\Omega\right)$
\[
\int_{\Omega}\nabla u\nabla\varphi\,dV=
\int_{\Omega}\lambda\left(x\right)f\left(u\right)\varphi+ h\varphi\,dV.
\]
\end{definition}

\bigskip
A $d$-slab is the region contained between two parallel hyperplanes at distance $d$.
Recall the notation
\[
\Pi_d=\left\{\left(x_1, \dots, x_{n-1}, x_n\right)\in \mathbb{R}^n\,:\, -\frac{d}{2}<x_{n}
<\frac{d}{2}\right\}.
\]
A $d$-thin domain is a domain contained in a $d$-slab. Our next theorem,
which is a version of Theorem \ref{thm:mainthm_discontinuous} for slabs,
can be stated for more general $d$-thin domains, however we have 
stated and prove it in the case of slabs to minimize technical difficulties.

\begin{theorem}
\label{thm:discontinuous_noncompact}
Let $f:\left[0,\infty\right)\longrightarrow \left[0,\infty\right)$ be a  
nondecreasing lower semicontinuous function such that
\[
f\left(s\right)\leq Ks^p,
\]
and let $h, \lambda\in L^{\infty}\left(\Pi_d\right)$, $h\leq 0$, $\lambda\geq 0$.
\begin{itemize}

\item
If $p<1$, there is always a weak solution $u\in L^{\infty}\left(\Pi_d\right)$ 
to BVP (\ref{eq: elliptic_prob_gen}) in the sense 
of Definition \ref{def:weak_solution}.

\item
If $p\geq 1$,
then for $h$ and $K$ small enough,
(\ref{eq: elliptic_prob_gen}) has a solution $u\in L^{\infty}\left(\Pi_d\right)$
in the sense 
of Definition \ref{def:weak_solution}.
Indeed, there is a solution as long as
\[
cK\Lambda\leq \frac{1}{2},\quad
\left\|h\right\|_{\infty}\leq  \Lambda K,
\]
where $c=\dfrac{d^2}{8}$. 
\end{itemize}
\end{theorem}

\begin{proof}
Consider the domain 
\[
\Omega_m = \Pi_d \cap \left[-m-2, m+2\right]^n.
\]
We impose Dirichlet boundary conditions in $\Omega_{m}$
and solve (\ref{eq: elliptic_prob_gen}) in $\Omega_{m}$
using the iteration procedure;
call each of these solutions $u_m$, which belong
to $W_0^{1,2}\left(\Omega_m\right)\cap L^{\infty}\left(\Omega_m\right)$. We have that 
in the case of sublinear growth of $f$ 
the solutions are
uniformly bounded, that the bound is independent of the domain,
as all lie within the same slab, and
in the case of linear or superlinear growth, under the hypotheses
of the theorem, the solutions produced via iteration satisfy
\[
\left\|u_m\right\|_{\infty}\leq 1, \quad
\mbox{for all} \quad m,
\]
and also that $u_{m+1}\geq u_{m}$ on $\Omega_m$. Let us prove this last claim.
Firstly, given $\Omega_m$, let $u_{m,k}$ be the sequence of approximations
produced to solve the Dirichlet problem: we show that
$u_{m+1,0}\geq u_{m,0}$. Indeed, notice that as $\partial \Omega_m \in \overline{\Omega}_{m+1}$,
then $u_{m+1,0}\geq 0$ on $\partial \Omega_m$. In $\Omega_m$ we have
\begin{eqnarray*}
\Delta\left(u_{m+1,0}-u_{m,0}\right)&=& 0,
\end{eqnarray*}
and thus $u_{m+1,0}\geq u_{m,0}$ by the maximum principle.
Assuming that $u_{m+1,k}\geq u_{m,k}$ we have that
\begin{eqnarray*}
\Delta\left(u_{m+1,k+1}-u_{m,k+1}\right)&=&
-\lambda\left(x\right)\left[f\left(u_{m+1,k}\right)-f\left(u_{m,k}\right)\right]\leq 0,
\end{eqnarray*}
and the classical maximum principle gives the claim.
Therefore, $u_{m+1}\geq u_m$,
and there is a limit $u=\sup u_m$,
which is a solution to the equation in the whole domain.

\end{proof}

Remark \ref{sect:remarksublinear} also holds for the previous theorem.

\bigskip
We also have the analogue of Theorem \ref{thm:mainthm_lipschitz}. We present its
statement and proof, followed by an interesting application to geometry.
\begin{theorem}
\label{thm:lipschitz_noncompact}
Consider the BVP (\ref{eq: elliptic_prob_gen}),
with $f$ differentiable, nonnegative, nondecreasing, and such that $f'$ is bounded 
on any compact subset of $\mathbb{R}$, and $\lambda\geq 0, h\geq 0$ both in
$L^{\infty}\left(\Pi_d\right)$. Let $0<\theta<1$,  
\[
L= \max_{x\in \left[-c\left\|h\right\|_{\infty}, c\left\|h\right\|_{\infty}\right]} f\left(x\right),
\]
and 
\[
M= M\left(\lambda\right):=\sup_{x\in \left[-c\left\|h\right\|_{\infty}-\frac{1}{\theta}\Lambda cL,
c\left\|h\right\|_{\infty}+\frac{1}{\theta}\Lambda c L\right]}
f'\left(x\right),
\]
where $c=\dfrac{d^2}{8}$, and $\Lambda=\left\|\lambda\right\|_{\infty}$.
Then for $\lambda$ such that
\[
c\Lambda M \leq 1-\theta,
\]
there is a solution $u \in L^{\infty}\left(\Pi_d\right)$
in the sense 
of Definition \ref{def:weak_solution},
and we can estimate
\[
\left\|u\right\|_{\infty}\leq 
c\left\|h\right\|_{\infty}+\frac{1}{\theta}c \Lambda L.
\]
\end{theorem}

\begin{proof}
Again we solve (\ref{eq: elliptic_prob_gen}) in $\Omega_m$, and call the solution
obtained $u_m$. To show that $u_{m+1}\geq u_m$ we proceed by induction.
First, notice that
$u_1\geq u_0$, this because $u_1\geq 0$ on $\Omega_1\supset \Omega_0$,
\[
\Delta\left(u_1-u_0\right)=-\lambda f\left(u_0\right)\leq 0,
\]
and thus the maximum principle gives the claim.
Next,
\[
\Delta\left(u_{m+1}-u_m\right)=-\lambda \left(f\left(u_m\right)-f\left(u_{m-1}\right)\right)=
-\lambda f'\left(\xi_{m}\right)\left(u_{m}-u_{m-1}\right),
\]
where $\xi_{m}\left(x\right)$ is an intermediate point between $u_{m}\left(x\right)$
and $u_{m-1}\left(x\right)$. Since $f$ is nondecreasing and $\lambda\geq 0$,
we have that $\lambda f'\left(\xi_m\right)\leq 0$, and since $u_{m+1}\geq 0$
on $\Omega_{m}$, the claim follows from the maximum principle. By Theorem
\ref{thm:mainthm_lipschitz}, we have a uniform bound on the sequence of $u_m$'s and hence there
is a limit $u=\sup_m u_m$, and this is a solution to the BVP (\ref{eq: elliptic_prob_gen})
in the sense of Definition \ref{def:weak_solution} on $\Pi_d$.

\end{proof}

 It is not difficult to show that whenever $\lambda$ and $h$ are H\"older continuous,
the solutions obtained in the proof of the previous theorem are classical solutions.

\subsection{A comment on conformal deformation of metrics a 2D-slab}
\label{sect:conformal_strip}
Theorem \ref{thm:lipschitz_noncompact} implies that in $\Pi_{2\sqrt{2}}\subset \mathbb{R}^2$ there is a metric of constant
positive Gaussian curvature $\lambda$, conformal to the flat
metric, as long as $\lambda\leq 0.1452\dots$. If $\Pi_{2\sqrt{2}}$
were boundaryless and endowed with the metric $g=e^{2u}g_E$, where $g_E$ is the euclidean
metric on $\Pi_{2\sqrt{2}}$, the Bonnet-Myers theorem,
if such metric is complete, would imply that the manifold must be compact. But then, notice
the following: if $\lambda\geq 0$ then the maximum principle implies that $u\geq 0$, and
it is not difficult to prove that as a metric space $\Pi_{2\sqrt{2}}$ with
the distance function induced by $g$ is complete. So,
what is going on here? If we revise the standard proof of the Bonnet-Myers theorem,
then it is clear that we need the existence of a point $p\in \Pi_{2\sqrt{2}}$ that can be joined
to any other point by an unbroken geodesic; hence, in this Riemannian
manifold with boundary, given a point there is another that cannot
be joined to it by an unbroken minimizing geodesic. Finally,
notice that by the results in \cite{Berestycki}, if $\lambda$ is constant, $u$ only depends on $x_2$.

\subsection{Nontrivial solutions of the Lane-Emden system in thin domains.} 
\label{sect:systems}
To show that the Dirichlet problem
\[
-\Delta u =u^p \quad \mbox{in}\quad \Omega, \quad u=0 \quad \mbox{on} \quad \partial \Omega,
\]
has a nontrivial solution in thin unbounded domains, we need
a couple of observations. First an extension to thin bounded 
domains. In this case, we shall assume the closure of the unit ball $\overline{D}_1\subset \Omega$.
We define a function $z$ as follows.

On $D_{1-\eta}$ for $\eta>0$ small enough 
\[
w\left(r\right)=\frac{\epsilon\left(n+4\right)}{4n\left(n+2\right)}
\left[\frac{n}{n+4}r^4-\frac{2\left(n+2\right)}{n+4}r^2+1\right], \quad
\mbox{on}\quad D_{1-\eta}, 
\] 
where $\epsilon$ is such that in $D_1$
(not just on $D_{1-\eta}$)
$w^p > \epsilon\left(1-r^2\right)$,
and $v\equiv 0$ on $\Omega\setminus D_{1+\eta'}$, with $\eta' >0$ also small enough.

Next, in the annulus $D_{1+\eta'}\setminus D_{1-\eta}$, define a rotationally symmetric function $\omega\geq 0$ such that at 
radius $1-\eta$ 
glues smoothly
(that is, at least up to the 
second derivative) to $w$, and at radius $1+\eta'$ glues smoothly to 0. What we require from
$\omega$ is that it is decreasing and $-\Delta \omega \leq \omega^p$ from
radius $1-\eta$ to radius $1$, and past radius $1$ and up to radius $1+\eta'$, $\Delta w\geq 0$. 
In fact, we can have $-\Delta \omega \leq -\Delta w$ and 
$w^p\leq \omega^p$ from radius $1-\eta$ to radius $1$ (see figure \ref{fig:profile}).

Thus, define $z$ as $w$ from radius 0 to radius $1-\eta$, as $\omega$ from radius 
$1-\eta$ to radius $1+\eta'$, and as 0 outside the ball $D_{1+\eta'}$.
For a given $p\leq \frac{1}{2}$, it is easy to show that 
$z$ satisfies $-\Delta v\leq v^p$.

From this we can set up an iteration to find
nontrivial solutions to the Dirichlet problem
with $p\leq 1/2$, 
in a bounded 
domain of slab diameter $d$. It is not difficult to check that these solutions can be 
explicitly bounded,
and the bound depends on the slab diameter of the domain.
Since the power functions $x^r$, $r>0$, are increasing, this allows us to use the
arguments used in the proofs of Theorems 
\ref{thm:discontinuous_noncompact} and \ref{thm:lipschitz_noncompact} to show that the
Dirichlet problem has a nontrivial solution in $\Omega=\Pi_d$. Again,
this argument holds for more general unbounded domains, as long as they are thin.

\begin{center}
\begin{figure}[H]
\begin{tikzpicture}

\begin{axis}[xmin=-2, xmax=2, ymin=-0.05, ymax=0.25,
 xtick = {-1.5,0,1.5}, ytick = { 1},
 scale=1, restrict y to domain=-1.5:1.2,
 axis x line=center, axis y line= center,
 samples=40
 ]

 \addplot[black, samples=100, smooth, domain=-1.0:1, thick]
   plot (\x, {(3/16)*((1/3)*x^4-(4/3)*x^2+1)});
   
 \addplot[smooth] coordinates {
   (0.8, 0.053)
   (0.81, 0.049)
   (0.82, 0.045)
   (0.85, 0.038)
   (0.9, 0.025)
   (1, 0.010)
   (1.05, 0.005)
   (1.07, 0.003)
   (1.1, 0.002)
   (1.15, 0.001)   
   (1.2, 0)
   (1.22, 0)
   (1.25, 0)
   };  
   \draw[arrows=<-] (1.1, 0.01)--(1.3, 0.02);
   \draw[arrows=<-] (0.7, 0.10)--(1.0, 0.11);
   \draw[arrows=<-] (1.201, 0.0005) -- (1.65, 0.012);
   \draw (0.8, 0.053)--(0.8, -0.005);
   \node at (1.35, 0.03) {$\omega$};
   \node at (1.1, 0.11) {$w$};
   \node at (0.8, -0.01) {\tiny{$1-\eta$}};
   \node at (1.8, 0.02) {\tiny{$1+\eta'$}};
\end{axis}

\end{tikzpicture}

\caption{Profile view of the function $z$.}\label{fig:profile}
\end{figure}
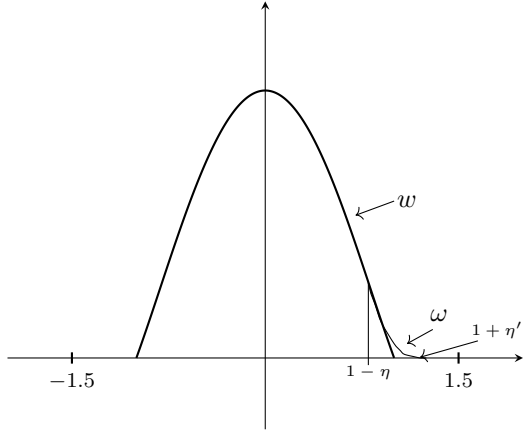
\end{center}

Using $z$ (assume of course that $\overline{D}_1\left(0\right)\subset \Omega$), we can also construct 
nontrivial solutions to systems of the form 
\[
-\Delta u = v^p, \quad -\Delta v = u^q, \quad u=v=0 \quad \mbox{on} \quad \partial \Omega, 
\]
as follows. If $p, q\leq 1/2$, we can choose $z$ so that 
\[
-\Delta z \leq z^p \quad \mbox{and}\quad -\Delta z\leq z^q.
\]
Define in $\Omega$ the iteration
\[
-\Delta u_{n+1}=v_{n}^{p}, \quad -\Delta v_{n+1}=u_{n}^q, \quad u_{n+1}=v_{n+1}=0 \quad \mbox{on}
\quad \partial \Omega,
\]
with $u_{0}=v_{0}=z$.

Using the maximum principle and induction it is not difficult to show that $u_{n+1}\geq u_{n}$
and $v_{n+1}\geq v_{n}$. From this it follows that
the sequence of solutions $\left(u_n, v_n\right)$ to the family of Dirichlet problems defined
in the iteration, as it can be shown to be bounded above, this bound only depending on the slab diameter of $\Omega$, 
converges to a nontrivial solution of the system.

This can be extended to thin unbounded domains. We sketch how to do it in the slab $\Pi_d$.
In $\Omega_{m}$ we can solve 
\[
-\Delta u_{m, n+1}=v_{m, n}^{p}, \quad -\Delta v_{m, n+1}=u_{m,n}^q, \quad u_{m,n+1}=v_{m, n+1}=0 \quad \mbox{on}
\quad \partial \Omega_{m},
\]
and here $u_{m,0}=v_{m,0}=z$. Again, it can be shown, using the maximum principle, that
$u_{m+1,n}\geq u_{m,n}$ and $v_{m+1,n}\geq v_{m,n}$, and from this we get that the solutions to the
Dirichlet problem 
\[
-\Delta u_m = v_m^p, \quad -\Delta v = u_m^q, \quad u_m=v_m=0 \quad \mbox{on} \quad \partial \Omega_m, 
\]
constructed via iteration, satisfy $u_{m+1}\geq u_m$ and $v_{m+1}\geq v_m$, and thus we can produce a nontrivial solution pair to the 
Dirichlet problem in $\Pi_d$.

\end{document}